\documentclass[a4paper,reqno,11pt,oneside]{amsart}
\usepackage[all]{xy}
\usepackage{cm}

\title{Derived autoequivalences of generalised Kummer varieties}
\author{Ciaran Meachan}
\address{Mathematisches Institut, Universit\"{a}t Bonn, Deutschland}
\email{meachan@math.uni-bonn.de}

\begin{document}

\begin{abstract}
 In this article, we construct new derived autoequivalences of generalised Kummer varieties. Together with Huybrechts-Thomas twists around $\bbP^n$-objects, these are the only known examples of such symmetries.
\end{abstract}

\maketitle
{\let\thefootnote\relax\footnotetext{This work was supported by the SFB/TR 45 `Periods, Moduli Spaces and Arithmetic of Algebraic Varieties' of the DFG (German Research Foundation).}}

\section*{Introduction}
Let $\cD(X)$ be the bounded derived category of coherent sheaves on a smooth complex projective variety $X$ and $\Aut(\cD(X))$ be the group of isomorphism classes of exact $\bbC$-linear autoequivalences of $\cD(X)$. Then we have a subgroup of \emph{standard} autoequivalences: \[\Aut(X)\ltimes(\Pic(X)\oplus\bbZ)\subset \Aut(\cD(X))\] generated by push forwards along automorphisms of $X$, twists by line bundles and shifts; the complement of this subgroup should be thought of as \emph{hidden symmetries}. 

In their seminal paper, Bondal and Orlov \cite[Theorem 3.1]{bondal2001reconstruction} showed that if the canonical bundle $\omega_X$ (or its inverse) is ample then the standard autoequivalences are everything, i.e. $\Aut(\cD(X))\simeq\Aut(X)\ltimes(\Pic(X)\oplus\bbZ)$. However, finding a complete description of this group when the canonical bundle is not ample is a much more subtle question and one of considerable interest. 
The case where the canonical bundle is trivial is particularly rich.

The classical decomposition theorem of Beauville \cite[Th\'eor\`eme 1]{beauville1983varietes} tells us that there are essentially three fundamental classes of varieties with trivial canonical bundle: abelian varieties, Calabi-Yau varieties and hyperk\"ahler varieties. When $X$ is an abelian variety, 
Orlov \cite[Theorem 4.14]{orlov2002derived} has given a complete description of $\Aut(\cD(X))$ but the other cases remain wide open. Despite not having a complete description of $\Aut(\cD(X))$ in the Calabi-Yau case, we do have examples of non-standard autoequivalences thanks to the pioneering work of Seidel and Thomas \cite{seidel2001braid} on \emph{spherical objects}. For example, the structure sheaf of a $(-1,-1)$-curve inside a Calabi-Yau threefold gives rise to an autoequivalence, called the spherical twist. It turns out that this can also be described as doing the Atiyah-flop equivalence \cite[Theorem 3.6]{bondal1995semiorthogonal} twice; for generalisations of this to $(-2,0)$-curves and $(-3,1)$-curves, see \cite[Theorem 3.1]{toda2007on} and \cite[Theorem 1.9]{donovan2013noncommutative} respectively. 

Finally, for hyperk\"ahler varieties, 
much less is known. For a K3 surface, which is the simplest hyperk\"ahler variety, we have a conjectural description of $\Aut(\cD(X))$ due to Bridgeland \cite[Conjecture 1.2]{bridgeland2008stability} phrased in terms of stability conditions; roughly speaking, the expectation is that $\Aut(\cD(X))$ should be generated by standard autoequivalences and spherical twists
; see \cite[Theorems 1.3 \& 1.4]{bayer2013derived} for the Picard rank one case. For higher dimensional hyperk\"ahler varieties, Huybrechts and Thomas \cite{huybrechts2006pobjects} generalised the notion of spherical objects to that of $\bbP^n$-objects. For example, the structure sheaf of an embedded $\bbP^n$ inside a hyperk\"ahler variety of dimension $2n$ gives rise to an autoequivalence, called the $\bbP^n$-twist. It is expected that there is a strong connection between this twist and doing the Mukai-flop equivalence \cite[Corollary 5.7]{kawamata2002dequivalence} \cite[Theorem 4.4]{namikawa2003mukai} twice. To date, Huybrechts-Thomas twists associated to the structure sheaf of the variety are still the only non-standard autoequivalences which can be associated to an arbitrary hyperk\"ahler. Recently, Addington \cite{addington2011new} further generalised this concept to $\bbP^n$-functors and constructed derived autoequivalences for Hilbert schemes of points on K3 surfaces which were evidently not equal to any of the known ones.

The aim of this paper is to extend this discovery and show that we also get new derived autoequivalences for generalised Kummer varieties.

\subsection*{Summary of main results}
 The generalised Kummer variety $K_n$ associated to an abelian surface $A$ is, by definition, the fibre of the Albanese map $m:A^{[n+1]}\ra A$ over zero where $A^{[n+1]}$ is the Hilbert scheme of $n+1$ points on $A$. 
 In particular, the 
 natural family $\cI$ of ideal sheaves on $A\times K_n$ gives rise to a Fourier-Mukai functor $F_K:\cD(A)\ra\cD(K_n)$. For all $n>1$, we show that $F_K$ is a $\bbP^{n-1}$-functor in the sense of Addington \cite[Section 3]{addington2011new}. That is, if $R_K$ denotes the right adjoint to $F_K$ then the kernel of the composition $R_KF_K$ is given by $\bigoplus_{i=0}^{n-1}\cO_\Delta[-2i]$ and the monad structure $R_KF_KR_KF_K\srel{\mu}{\ra}R_KF_K$ closely resembles multiplication in the graded ring $H^*(\bbP^{n-1},\bbC)$. Addington's $\bbP^n$-functors are a generalisation Huybrechts and Thomas' $\bbP^n$-objects \cite{huybrechts2006pobjects} and as such, they determine autoequivalences of the codomain category. For instance, when $A$ is an abelian surface and $F_K$ is the functor described above, a double cone construction produces a non-trivial element $P_{F_K}\in\Aut(\cD(K_n))$ which acts on $\im F_K$ by $[2-2n]$ and $(\im F_K)^\perp$ by the identity; the induced action on cohomology is trivial. In section \ref{genkum}, we prove our main
 \begin{thm*}[\ref{kummer}]
  $F_K$ is a $\bbP^{n-1}$-functor for all $n>1$. In particular, we have a new derived autoequivalence \[P_{F_K}:=\cone(\cone(F_KR_K[-2]\to F_KR_K)\to\id_{K_n})\in\Aut(\cD(K_n)).\]
\end{thm*}

 The key observation in proving this result 
 is the following
 \begin{thm*}[\ref{Albanese}]
  Let $m:A^{[n]}\ra A$ be the Albanese map. Then $m^*:\cD(A)\to\cD(A^{[n]})$ is a $\bbP^{n-1}$-functor. Thus, we generate new derived autoequivalences of the Hilbert scheme $A^{[n]}$ as well. 
 \end{thm*}

In section \ref{equiv}, we adopt an equivariant approach and study tautological objects on the generalised Kummer variety. In particular, if $F_K'':\cD(A)\to\cD(K_n)$ denotes the Fourier-Mukai transform induced by the structure sheaf of the universal subscheme inside $A\times K_n$ then we have 
 \begin{thm*}[\ref{extformulakummer}]
 Let $A$ be an abelian surface and consider the tautological objects $F_K''(\cE), F_K''(\cF)\in\cD(K_n)$ associated to $\cE,\cF\in\cD(A)$. Then we have the following natural isomorphism of graded vector spaces
 \begin{eqnarray*}\Ext^*(F_K''(\cE),F_K''(\cF)) &\simeq& \Ext^*(\cE,\cF)\otimes H^*(K_{n-1},\cO_{K_{n-1}})\\ && \oplus\; H^*(A,\cE^\vee)\otimes H^*(A,\cF) \otimes H^*(K_{n-2},\cO_{K_{n-2}}).
 \end{eqnarray*}
\end{thm*}

Finally, if $i:K_n\hra A^{[n+1]}$ denotes the inclusion, $j:N\hra A^{n+1}$ is the locus of points which sum to zero, and $D$ is the union of pairwise diagonals in $A\times A^{n+1}$, then we obtain the following
\begin{thm*}[\ref{BKRtype}]
There is a `BKR-type' equivalence $\Psi_K:\cD(K_n)\xra\sim\cD^{\frS_{n+1}}(N)$ which naturally intertwines with $\Psi:\cD(A^{[n+1]})\xra\sim\cD^{\frS_{n+1}}(A^{n+1})$. That is, we have 
\[\Psi_K\circ i^*\simeq j^*\circ\Psi.\]
In particular, this provides us with the identities $\Psi_KF_K'\simeq\Phi_{\cO_{A\times N}}$, $\Psi_KF_K''\simeq\Phi_{\bar{D}}$ as well as an $\frS_{n+1}$-equivariant resolution $\bar{\cK}^\bullet$ of $\bar{D}:=(\id_A\times j)^{-1}D$ given by \[0\to\cO_{\bar{D}}\to\bigoplus_{i=1}^{n+1}\cO_{\bar{D}_i}\to\cdots\ra\bigoplus_{|I|=k}\cO_{\bar{D}_I}\to\cdots\to\cO_{\bar{D}_{\{1,\ldots,n+1\}}}\to0\] where $\bar{D}_I:=(\id_A\times j)^{-1}(D_I)$ are the restricted diagonals.
 \end{thm*}
 
 As a byproduct of this approach, we provide alternative proofs of Theorem \ref{kummer} above and \cite[Theorem 2]{addington2011new}. 
\vs

 \textbf{Acknowledgements}: The author is profoundly grateful to Daniel Huybrechts; not only for the invitation to Bonn but also for all his valuable help and support. Special thanks to Nicolas Addington for patiently explaining the details of \cite{addington2011new}; Andreas Krug for generously donating his time to discussions regarding applications of \cite{krug2014extension}; Arend Bayer together with Will Donovan for numerous technical conversations; and Eyal Markman \& Sukhendu Mehrotra for sharing a preliminary version of \cite{markman2011integral}. The referees have been extremely thorough and their suggestions have improved the paper considerably. 
\vs

\section{$\bbP$-functors}\label{defs}
\begin{dfn}
 An exact functor $F:\cA\ra\cB$ between triangulated categories with left and right adjoints $L,R:\cB\ra\cA$ is a $\bbP^n$-functor if the following conditions are satisfied:
\begin{itemize}
 \item[(i)] There is an autoequivalence $H$ of $\cA$ such that $$RF\simeq \id\oplus H\oplus H^2\oplus\cdots\oplus H^n$$
 \item[(ii)] The map $HRF\hra RFRF\srel{R\epsilon F}{\lra} RF$, when written in components $$H\oplus H^2\oplus\cdots\oplus H^n\oplus H^{n+1}\ra \id\oplus H\oplus H^2\oplus\cdots\oplus H^n,$$ is of the form
 $$\left(\begin{matrix}\ast & \ast & \cdots & \ast & \ast\\ 1 & \ast & \cdots & \ast & \ast\\ 0 & 1 & \cdots & \ast & \ast\\ \vdots & \vdots & \ddots & \vdots & \vdots\\ 0 & 0 & \cdots & 1 & \ast\end{matrix}\right)$$
 \item[(iii)] $R\simeq H^n L$. If $\cA$ and $\cB$ have Serre functors, this condition is equivalent to $S_\cB F H^n\simeq FS_\cA$.
\end{itemize}
\end{dfn}

\begin{thm}[{\cite[Theorem 3]{addington2011new}}]
 If $F$ is a $\bbP^n$-functor then \[P_F:=\cone(\cone(FHR\srel{f}{\ra}FR)
 {\ra}\id_\cB)\] is an equivalence where $f$ is the composition $FHR\hra FRFR\xra{\epsilon FR-FR\epsilon} FR$.
\end{thm}


\begin{exas}\label{examples}
\begin{enumerate}
 \item Let $S$ be a smooth projective K3 surface and consider the natural functor $F:\cD(S)\ra\cD(S^{[n]})$ induced by the universal ideal sheaf $\cI$ on $S\times S^{[n]}$. Then $F$ is a $\bbP^{n-1}$-functor with $RF\simeq\id\oplus[-2]\oplus\cdots\oplus[2-2n]$ and $H\simeq[-2]$ \cite[Theorem 2]{addington2011new}. It was precisely this example which inspired the author to consider Beauville's \cite{beauville1983varietes} other infinite family of irreducible holomorphic symplectic manifolds: the generalised Kummer variety.
 \item A split spherical functor $F:\cA\ra\cB$ is one where the exact triangle $\id_\cA\srel{\eta}{\ra}RF\ra C$ is split, i.e. $RF\simeq \id_\cA\oplus C$. 
 In other words, a split spherical functor is a $\bbP^1$-functor with $H\simeq C$. Just as in \cite[Proposition 2.9]{huybrechts2006pobjects}, the $\bbP^1$-twist $P_F$ associated to a split spherical functor is equivalent to the square of the spherical twist $T:=\cone(FR\srel{\epsilon}{\ra}\id_\cB)$. See \cite[p.37]{addington2011new}.
\end{enumerate}
\end{exas}

\section{Nested Hilbert Schemes}
 The key results needed for the calculation of $RF$ in the Hilbert scheme setting come from Ellingsrud and Str\o mme's work \cite{ellingsrud1998an} on nested Hilbert schemes on smooth projective surfaces. Let us consider the following diagram:
 $$\xymatrix{A^{[n,n+1]} \ar[r]^-g \ar[d]_-{q\times f} & A^{[n+1]}\\ A\times A^{[n]} &}$$
 where $A^{[n,n+1]}:=\{(\zeta,\xi)\in A^{[n]}\times A^{[n+1]}\;|\;\zeta\subset\xi\}$ is the incidence variety, $$g:A^{[n,n+1]} \ra A^{[n+1]}\;;\;(\zeta,\xi)\mapsto \xi\quad\trm{and}\quad f:A^{[n,n+1]} \ra A^{[n]}\;;\;(\zeta,\xi)\mapsto \zeta$$ are the natural maps induced by the projections and $$q:A^{[n,n+1]} \ra A\;;\;(\zeta, \xi)\mapsto\xi\backslash\zeta:= \Supp(\ker(\cO_\xi\ra\cO_\zeta))$$ maps a pair of subschemes to the point where they differ.

\begin{prop}\cite[Proposition 2.1 \& 2.2]{ellingsrud1998an}
\begin{itemize}
 \item[(i)] The map $g:A^{[n,n+1]}\ra A^{[n+1]}$ factors naturally over the universal subscheme $\cZ_{n+1}\subset A\times A^{[n+1]}$ as $g=\pi_{n+1}\circ\psi$ where $\pi_{n+1}:\cZ_{n+1}\ra A^{[n+1]}$ is the restriction of the projection and $\psi:A^{[n,n+1]}\ra\cZ_{n+1}$ is canonically isomorphic to $\bbP(\omega_{\cZ_{n+1}})$. In particular, $\psi$ is birational and an isomorphism over the set $\{(x,\xi)\in\cZ_{n+1}:\xi\trm{ is a local complete intersection at }x\}$
 , and $g$ is generically finite of degree $n+1$. 
 \item[(ii)] The map $q\times f:A^{[n,n+1]}\ra A\times A^{[n]}$ is canonically isomorphic to the blowup of $A\times A^{[n]}$ along $\cZ_{n}$. In particular, over $\cZ_n'$, the map $q\times f$ is a $\bbP^1$-bundle.
\end{itemize}
\end{prop}

 For any point $(\zeta,\xi)\in A^{[n,n+1]}$ we have two natural short exact sequences on $A$: $$0\ra\cI_\xi\ra\cI_\zeta\ra\cO_{\xi\setminus\zeta}\ra0\quad \trm{and}\quad 0\ra\cO_{\xi\setminus\zeta}\ra\cO_\xi\ra\cO_\zeta\ra0.$$ Using these, we see that the fibre of $(q\times g)$ over a point $(x,\xi)\in A\times A^{[n+1]}$ is the projective space $\bbP\Hom(\cO_x,\cO_\xi)^*$ and the fibre of $(q\times f)$ over a point $(x,\zeta)\in A\times A^{[n]}$ is the projective space $\bbP\Hom(\cI_\zeta,\cO_x)^*\simeq\bbP(\cI_\zeta|_x)$. In particular, we have $$(q\times g)_*\cO_{A^{[n,n+1]}}\simeq\cO_{\cZ_{n+1}}\quad \trm{and}\quad (q\times f)_*\cO_{A^{[n,n+1]}}\simeq\cO_{A\times A^{[n]}}.$$ It can also be shown that the exceptional divisor $$E:=(q\times f)^{-1}(\cZ_{n+1})=\{(\zeta,\xi)\in A^{[n,n+1]}\;|\;(\xi\backslash\zeta)\subset\zeta\}$$ is irreducible \cite[Section 3]{ellingsrud1998an}, $\omega_{A^{[n,n+1]}}\simeq\cO(E)$, 
 $(q\times f)_*\cO(E)\simeq\cO_{A\times A^{[n]}}$, 
 $(q\times f)_*\cO_E(E)=0$ and $q$ is a submersion \cite[Section 2.1]{addington2011new}.

\section{Hilbert Schemes of Points on an Abelian Surface}\label{HilbA}
 The Hilbert scheme $A^{[n+1]}$ of $n+1$ points on an abelian surface $A$ can be thought of as a fine moduli space of ideal sheaves on $A$ with trivial determinant. In particular, the structure sequence $0\ra\cI_{\cZ_{n+1}}\ra \cO_{A\times A^{[n+1]}}\ra \cO_{\cZ_{n+1}}\ra0$ for the universal subscheme $\cZ_{n+1}\subset A\times A^{[n+1]}$ gives rise to a sequence of Fourier-Mukai functors $$F\quad F'\quad F''\;:\;\cD(A)\ra\cD(A^{[n+1]})$$ whose right adjoints are denoted by $R,R',R''$ respectively. It is the observation that $F''\simeq g_*q^*$ which brings the nested Hilbert scheme into play. Now, performing a similar calculation to that of \cite[Sections 2.2 \& 2.3]{addington2011new} shows that the kernels of the compositions are:
\begin{eqnarray*}
 R'F' & \simeq & \cO_{A \times A}[2] \oplus \bigoplus_{i=-1}^{2n-1} \cO_{A \times A}^{\oplus 2}[-i] \oplus \cO_{A \times A}[-2n]\\
 R'F''\;\;\simeq\;\; R''F'[2] & \simeq & \cO_{A \times A}[2] \oplus \bigoplus_{i=-1}^{2n-3} \cO_{A \times A}^{\oplus 2}[-i] \oplus\cO_{A \times A}[2-2n]\\
 R'F\;\;\simeq\;\; RF'[-2n] & \simeq & \cO_{A \times A}[2-2n] \oplus \cO_{A \times A}^{\oplus 2}[1-2n] \oplus \cO_{A \times A}[-2n]\\
 R''F'' & \simeq & \cO_\Delta \oplus \bigoplus_{i=1}^{2n-1} \cO_\Delta^{\oplus 2}[-i] \oplus \cO_\Delta[-2n]\\ && \oplus\;\cO_{A \times A} \oplus \bigoplus_{i=1}^{2n-3} \cO_{A \times A}^{\oplus 2}[-i] \oplus\cO_{A \times A}[2-2n]\\
 RF &  \simeq & \cO_\Delta \oplus \bigoplus_{i=1}^{2n-1} \cO_\Delta^{\oplus 2}[-i] \oplus \cO_\Delta[-2n]
\end{eqnarray*}
 where the penultimate line uses the fact that the triangle $$\cO_\Delta\otimes H^*(A^{[n]},\cO_{A^{[n]}}) \ra R''F'' \ra \cO_{A\times A}\otimes H^*(A^{[n-1]},\cO_{A^{[n-1]}})$$ splits because 
the extension class parametrising such triangles is, by construction, invariant under automorphisms of $A$. In particular, if $\iota:A\xra\sim A\;;\;x\mapsto-x$ is the involution then any class $e\in \Ext^1(\cO_{A\times A},\cO_\Delta)\simeq H^1(A,\cO_A)$ must satisfy $\iota^*e=e$. Since $\iota$ acts on $H^1(A,\cO_A)$ as $-\id_A$ we see that $e$ must be zero.

After Example \ref{examples}(2), one might have expected $F:\cD(A)\to\cD(A^{[n+1]})$ to also be a $\bbP^n$-functor but our calculation above shows this is not the case. 
For instance, when $n=1$, we have \begin{equation}\label{Hilb2}H\simeq C :=\cone(\id\srel{\eta}{\ra}RF) \simeq \cO_\Delta^{\oplus 2}[-1] \oplus \cO_\Delta[-2]\end{equation} 
which 
is not an autoequivalence of $\cD(A)$. Therefore, $F$ cannot be a $\bbP^n$-functor because $RF$ has the wrong `shape'. It is precisely our understanding of this example which allows us to see why restricting to the generalised Kummer makes the construction work. Notice, however, that we cannot work with nested generalised Kummer varieties because the natural incidence variety has the wrong dimension. Thus, we are forced to work with the functors above whilst keeping track of a specific subvariety; namely, Beauville's generalised Kummer variety.


\section{Generalised Kummer varieties}\label{genkum}
 The difference with the generalised Kummer variety $K_n$ is that it can be thought of as a fine moduli space of ideal sheaves on $A$ with trivial determinant \emph{and} trivial determinant of the Fourier-Mukai transform (with respect to the Poincar\'e bundle) \cite[Section 4]{yoshioka2001moduli}; this extra condition has the effect of killing the unwanted 
factor $\cO_\Delta^{\oplus2}$ of $H$ in \eqref{Hilb2}. If $i:K_n\hra A^{[n+1]}$ denotes the inclusion then the universal ideal sheaf $\cI_{\bar{\cZ}_{n+1}}:=(\id_A\times i)^*\cI_{\cZ_{n+1}}$ on $A \times K_n$ gives rise to a natural functor $$F_K:\cD(A)\ra\cD(K_n)$$ which is related to $F:\cD(A)\ra\cD(A^{[n+1]})$ in the following way: $$F_K \simeq i^*\circ F \quad\trm{and}\quad R_K \simeq R\circ i_*.$$ 
 
We now explain how the proof of \cite[Theorem 2]{addington2011new} can be adapted to yield
\[R_K'F_K' \srel{(\textbf{1})}{\simeq} \bigoplus_{i=-1}^{n-1}\cO_{A \times A}[-2i]\qquad R_K'F_K'' \srel{(\textbf{2})}{\simeq} \bigoplus_{i=-1}^{n-2}\cO_{A \times A}[-2i]\]
\[R_K'F_K \srel{(\textbf{3})}{\simeq} \cO_{A \times A}[2-2n]\qquad R_K''F_K' \srel{(\textbf{4})}{\simeq} \bigoplus_{i=0}^{n-1}\cO_{A \times A}[-2i]\qquad R_KF_K' \srel{(\textbf{5})}{\simeq} \cO_{A \times A}[2]\]
\[R_K''F_K'' \srel{(\textbf{6})}{\simeq} \bigoplus_{i=0}^{n-1}\cO_\Delta[-2i]\oplus \bigoplus_{i=0}^{n-2}\cO_{A \times A}[-2i]\]
 \[R_KF_K \srel{(\textbf{7})}{\simeq} \bigoplus_{i=0}^{n-1}\cO_\Delta[-2i]\]
 
(\textbf{1}) The composition is given by
\begin{eqnarray*}
 R_K'F_K'(\cE) & := & R'i_*i^*F'(\cE)\\
& \simeq & R'(i_*\cO_{K_n})\otimes H^*(A,\cE)\\
& \simeq & \cO_A\otimes H^*(A^{[n+1]},i_*\cO_{K_n}) \otimes H^*(A,\cE)[2]\\
& \simeq & \cO_A\otimes H^*(K_n,\cO_{K_n}) \otimes H^*(A,\cE)[2]\\
& \simeq & (\cO_A[2] \oplus \cO_A \oplus\cdots\oplus \cO_A[2-2n]) \otimes H^*(A,\cE)
\end{eqnarray*}
and so its kernel must be $$R_K'F_K'  \simeq  \cO_{A\times A}[2]\oplus \cO_{A\times A}\oplus\cdots\oplus\cO_{A\times A}[2-2n].$$
 
(\textbf{2}) In order to verify $R_K'F_K''$, one first needs to make the following observation. Consider the diagram:
 $$\xymatrix{ && A^{[n,n+1]} \ar[d]^-{q\times f} \ar[dll]_-{q} \ar[rr]^{g} &\ar@{}[d]|\circlearrowleft & A^{[n+1]} \ar[d]^-{m} && K_n \ar@{_{(}->}[ll]_-{i} \ar[d]\\  A && A\times A^{[n]} \ar[ll]_-{\quad \pi_1} \ar[rr]^{h} & & A && e \ar@{_{(}->}[ll]}$$
 where $h:=\Sigma\circ(\id_A\times m)$ and $\Sigma:A\times A\ra A$ is the addition map. Then, we have
\begin{eqnarray*}
 q_*g^*i_*\cO_{K_n} & \simeq & \pi_{1*}(q\times f)_*g^*i_*\cO_{K_n}\quad\trm{since }q\simeq\pi_1\circ(q\times f)\\
 &\simeq& \pi_{1*}(q\times f)_*g^*m^*\cO_e\quad\trm{since }i_*\cO_{K_n}\simeq m^*\cO_e\\
 &\simeq& \pi_{1*}(q\times f)_*(q\times f)^*h^*\cO_e\quad\trm{since }m\circ g = h\circ (q\times f)\\
 &\simeq& \pi_{1*}(h^*\cO_e\otimes (q\times f)_*\cO_{A^{[n,n+1]}})\quad\trm{by projection formula}\\
 &\simeq& \pi_{1*}(h^*\cO_e\otimes \cO_{A\times A^{[n]}})\quad\trm{since }(q\times f)_*\cO_{A^{[n,n+1]}}\simeq\cO_{A\times A^{[n]}}\\
 &\simeq& \pi_{1*}\cO_J
\end{eqnarray*}
 where $$J:=\Supp(h^*\cO_e)=h^{-1}(e)=\{(-m(\zeta),\zeta)\in A\times A^{[n]}\}\cong A^{[n]}$$ 
 In other words, $\pi_1|_J:J\ra A$ can be identified with $-m:A^{[n]}\ra A$ and so the required identity $\pi_{1*}\cO_J\simeq \bigoplus_{i=0}^{n-1}\cO_A[-2i]\simeq \cO_A\otimes H^*(K_{n-1},\cO_{K_{n-1}})$ follows from Lemma \ref{splitting} below. 
 Putting this all together, we see that
\begin{eqnarray*}
 R_K'F_K''(\cE) & := & R'i_*i^*F''(\cE)\\
 & \simeq & R'(i_*i^*g_*q^*(\cE))\\
 & \simeq & \cO_A\otimes H^*(A^{[n+1]},i_*i^*g_*q^*(\cE))[2]\\
 & \simeq & \cO_A\otimes H^*(A^{[n+1]},g_*q^*(\cE)\otimes i_*\cO_{K_n})[2]\quad\trm{by projection formula}\\
 & \simeq & \cO_A\otimes H^*(A^{[n,n+1]},q^*(\cE)\otimes g^*i_*\cO_{K_n})[2]\\
 & \simeq & \cO_A\otimes H^*(A,q_*(q^*(\cE)\otimes g^*i_*\cO_{K_n}))[2]\\
 & \simeq & \cO_A\otimes H^*(A,\cE\otimes q_*g^*i_*\cO_{K_n})[2]\quad\trm{by projection formula}\\
 & \simeq & \cO_A\otimes H^*(A,\cE\otimes \pi_{1*}\cO_J)[2]\quad\trm{since }q_*g^*i_*\cO_{K_n}\simeq\pi_{1*}\cO_{J}\\
 & \simeq & \cO_A\otimes H^*(K_{n-1},\cO_{K_{n-1}})\otimes H^*(A,\cE)[2]\\
 & \simeq & (\cO_A[2]\oplus\cO_A\oplus\cdots\oplus\cO_A[4-2n])\otimes H^*(A,\cE)
\end{eqnarray*}
and the kernel is given by $$R_K'F_K'' \simeq \cO_{A\times A}[2]\oplus\cO_{A\times A}\oplus\cdots\oplus\cO_{A\times A}[4-2n].$$
 
(\textbf{3}) To get our hands on $R_K'F_K$, we take cohomology of the natural triangle $$R_K'F_K\ra R_K'F_K' \ra R_K'F_K''$$ which produces the following long exact sequences 
$$0\ra \cH^i(R_K'F_K) \ra \cO_{A\times A} \ra \cO_{A\times A} \ra \cH^{i+1}(R_K'F_K) \ra 0$$  for each $i=-2,0,\ldots,2n-4$ as well as an identification $\cH^{2n-2}(R_K'F_K) \xra\sim \cO_{A\times A}$. In other words, if one can show  that the map $R_K' F_K' \to R_K' F_K''$ induces an isomorphism (or indeed a non-zero map) on $\cH^i$ for all $-2\leq i\leq 2n-4$ then $$\cH^i(R_K'F_K)=\left\{\begin{array}{cc}\cO_{A\times A} & \trm{if }i=2n-2\\ 0 & \trm{o/w}\end{array}\right.$$ and the kernel must be $$R_K'F_K \simeq \cO_{A\times A}[2-2n].$$
Given that the map of functors $R_K'F_K' \ra R_K'F_K''$ corresponds to the following map $\pi_{13*}(\pi_{12}^*\cO_{A\times K_n}\otimes\pi_{23*}\cO^\vee_{K_n\times A}[2])\ra \pi_{13*}(\pi_{12}^*\cO_{\bar{\cZ}_{n+1}}\otimes\pi_{23*}\cO^\vee_{K_n\times A}[2])$ of kernels, whose cohomology is given by $$\cE xt^{i+2}_{\pi_{13}}\left(\pi_{23}^*\cO_{K_n\times A},\pi_{12}^*\cO_{A\times K_n}\right)\ra \cE xt^{i+2}_{\pi_{13}}(\pi_{23}^*\cO_{K_n\times A},\pi_{12}^*\cO_{\bar{\cZ}_{n+1}}),$$ we see that a nonzero map here for $i=-2,0,\ldots,2n-4$ is equivalent to a nonzero map $$\cH^i(\pi_{1*}\cO_{A\times K_n})\ra \cH^i(\pi_{1*}\cO_{\bar{\cZ}_{n+1}})\quad\trm{for }i=0,2,\ldots,2n-2.$$ 
Since $\cO_{\bar{\cZ}_{n+1}}:=(\id_A\times i)^*\cO_{\cZ_{n+1}}\simeq(\id_A\times i)^*(q\times g)_*\cO_{A^{[n,n+1]}}$, we can consider the following diagram
$$\xymatrix{
&& A^{[n,n+1]} \ar[rr]^-{q\times g} \ar[d]_{q\times f} && A\times A^{[n+1]} \ar[d]^{\pi_1} && A\times K_n \ar[ll]_-{\id_A\times i} \ar[dll]^{\pi_1}\\ 
A\times K_{n-1} \ar[rr]^{\id_A\times i} && A\times A^{[n]} \ar[rr]_-{\pi_1} && A &&
}$$
and check the condition on the fibres of $q$ (which are all smooth); see \cite[p.21]{addington2011new}. That is, for $x\in A$ we want to show that $g^*|_{K_n}:H^i(\cO_{K_n})\xra\sim H^i(\cO_{(q\times f)^{-1}(x\times K_{n-1})})$ for all $0\leq i\leq 2n-2$. If $\sigma_{n+1}$ and $\sigma_{n}$ denote the natural symplectic forms on $A^{[n+1]}$ and $A^{[n]}$ induced by a symplectic form $\sigma$ on $A$ \cite[Proposition 5]{beauville1983varietes} then we have $g^*\sigma_{n+1}=q^*\sigma+f^*\sigma_n$ and $\bar\sigma_{n+1}^j\in H^{2j}(\cO_{A^{[n+1]}})$ is mapped to $f^*\bar\sigma_n^j\in H^{2j}(\cO_{q^{-1}(x)})$. Since $\sigma_{n+1}$ (respectively $\sigma_n$) induces a symplectic form on $K_n$ (respectively $K_{n-1}$) \cite[Proposition 7]{beauville1983varietes} and $f_*\cO_{(q\times f)^{-1}(x\times K_{n-1})}\simeq i^*f_*\cO_{q^{-1}(x)}\simeq i^*\cO_{A^{[n]}}\simeq\cO_{K_{n-1}}$ we see that $f^*|_{K_{n-1}}:H^i(\cO_{K_{n-1}})\xra\sim H^i(\cO_{(q\times f)^{-1}(x\times K_{n-1})})$ is an isomorphism and the generator of $H^{2j}(\cO_{K_n})$ must get mapped to the generator of $H^{2j}(\cO_{(q\times f)^{-1}(x\times K_{n-1})})$. 

 (\textbf{4}) 
Observe that $R_K''F_K'$ is the right adjoint of $L_K'F_K''$ and so 
 \begin{eqnarray*}
 R_K''F_K' &\simeq & (L_K' F_K'')^\vee [2]\quad\trm{since $R_K''F_K'\simeq S_{\cD(A)}^{}(L_K' F_K'')^\vee$}\\
 && \hspace{3cm}\trm{by \cite[Remark 5.8]{huybrechts2006fourier}}\\
 &\simeq & (R_K'[2n-2] F_K'')^\vee [2]\quad\trm{since $L_K'\simeq S_{\cD(A)}^{-1}R_K'S_{\cD(K_n)}^{}\simeq R_K'[2n-2]$}\\
&&\hspace{4.3cm}\trm{by \cite[Remark 1.31]{huybrechts2006fourier}} \\
 &\simeq & (R_K'F_K'')^\vee [4-2n]\\
 &\simeq & \cO_{A\times A}\oplus\cO_{A\times A}[-2]\oplus\cdots\oplus\cO_{A\times A}[2-2n].
 \end{eqnarray*}

(\textbf{5}) Similarly, the adjunction $L_K'F_K\dashv R_KF_K'$ provides us with 
\begin{eqnarray*}
 R_KF_K' &\simeq & (R_K'F_K)^\vee [4-2n]\\
 &\simeq & \cO_{A\times A}[2].
 \end{eqnarray*}

(\textbf{6}) The most technical calculation is that of $R_K''F_K''$. 
First notice that the kernel of the composition \cite[Proposition 5.10]{huybrechts2006fourier} is given by
\begin{eqnarray*}
R_K''F_K'' &\simeq& \pi_{13*}(\id_A\times i\times \id_A)^*(\cO_{\cZ_{n+1}\times A} \otimes \cO_{A\times \cZ_{n+1}}^\vee)[2]\\
&\simeq& \pi_{13*}(\cO_{\cZ_{n+1}\times A} \otimes \cO_{A\times \cZ_{n+1}}^\vee \otimes \pi_2^* i_* \cO_{K_n})[2]
\end{eqnarray*}
where $\cO_{\cZ_{n+1}\times A} \otimes \cO_{A\times \cZ_{n+1}}^\vee$ is supported on $(\cZ_{n+1} \times A) \cap (A \times \cZ_{n+1}) = \cZ_{n+1} \times_{A^{[n+1]}} \cZ_{n+1}$. Now, consider the maps
\[q \times g \times q: A^{[n,n+1]} \to A \times A^{[n+1]} \times A\;;\;(\zeta, \xi) \mapsto (\xi\backslash\zeta, \xi, \xi\backslash\zeta),\]
\[r: A^{[n-1,n,n+1]} \to A \times A^{[n+1]} \times A\;;\;(\eta, \zeta, \xi) \mapsto (\zeta\backslash\eta, \xi, \xi\backslash\zeta).\]
Then Addington \cite[p.23]{addington2011new} shows that there is an exact triangle
\begin{eqnarray}\label{triangle}
(q \times g \times q)_* \cO_{A^{[n,n+1]}}\to \cO_{\cZ_{n+1}\times A} \otimes \cO_{A\times \cZ_{n+1}}^\vee\to r_* \pr_{23}^* \cO(E),
\end{eqnarray}
where $\pr_{23}: A^{[n-1,n,n+1]} \to A^{[n,n+1]}\;;\;(\eta, \zeta, \xi) \mapsto (\zeta,\xi)$ is the projection. It is clear from the expression for $R_K''F_K''$ above that we want to pull this sequence back to $A\times K_n\times A$ and then push it down to $A \times A$. 

For the first term of \eqref{triangle}, observe that the following diagram commutes:
\[\xymatrix{
&& A^{[n+1]} && K_n\ar[ll]_-i\\
A^{[n,n+1]} \ar[d]_{q} \ar[rr]^-{q\times g\times q} \ar[urr]^-g && A \times A^{[n+1]} \times A \ar[d]^{\pi_{13}} \ar[u]_-{\pi_2} && A\times K_n\times A \ar[dll]^{\pi_{13}} \ar[ll]_-{\id_A\times i\times\id_A} \ar[u]_-{\pi_2}\\
A \ar[rr]_-{\Delta} && A \times A&&} \]
We have seen that $q_*g^*i_*\cO_{K_n} \simeq \cO_A \otimes H^*(K_{n-1},\cO_{K_{n-1}})$ and so the first term becomes 
\begin{eqnarray*}
\pi_{13*} ( (q \times g \times q)_* \cO_{A^{[n,n+1]}} \otimes \pi_2^* i_* \cO_{K_n} ) &\simeq& \pi_{13*} (q \times g \times q)_* (q \times g \times q)^* \pi_2^* i_* \cO_{K_n}\\
&&\hspace{2cm}\trm{by projection formula}\\
&\simeq& \Delta_* q_* g^* i_* \cO_{K_n}\\
&\simeq& \Delta_* (\cO_A \otimes H^*(K_{n-1},\cO_{K_{n-1}}))\\
&\simeq& \cO_\Delta \otimes H^*(K_{n-1},\cO_{K_{n-1}}).
\end{eqnarray*}

For the last term of \eqref{triangle}, we use a similar diagram (and temporarily introduce another map $q_{23}:A^{[n-1,n,n+1]}\to A\;;\;(\eta,\zeta,\xi)=\xi\backslash\zeta$)
\[\resizebox{\columnwidth}{!}{ \xymatrix{A&& A^{[n,n+1]} \ar[rr]^-g \ar[dll]_(0.275){q\times f} && A^{[n+1]} && K_n \ar[ll]_-i\\
A\times A^{[n]}\ar[u]^-h && A^{[n-1,n,n+1]} \ar[rr]^-r \ar[u]_-{\pr_{23}} \ar[dll]_-{q_{23}\times \pr_{12}\;\;} && A \times A^{[n+1]} \times A \ar[d]^-{\pi_{13}} \ar[u]_-{\pi_2} && A\times K_n\times A \ar[ll]_-{\id_A\times i\times\id_A} \ar[dll]^-{\pi_{13}} \ar[u]_-{\pi_2}\\ 
A\times A^{[n-1,n]} \ar[u]^-{\id_A\times g} \ar[rr]^-{\id_A\times q\times f} && A\times A\times A^{[n-1]} \ar[rr]^-{\pi_{21}} \ar@/_1.7pc/[uull]_(0.8){\Sigma\circ(\id_A\times h)} && A\times A &&}}\]
\begin{eqnarray*}
&& \pi_{13*} ( r_*\pr_{23}^* \cO(E) \otimes \pi_2^* i_* \cO_{K_n} ) \\
&\simeq& \pi_{13*}r_* ( \pr_{23}^* \cO(E) \otimes r^*\pi_2^* i_* \cO_{K_n} )\quad\trm{by projection formula}\\
&\simeq& \pi_{21*}(\id_A\times q\times f)_*(q_{23}\times\pr_{12})_*\pr_{23}^*( \cO(E) \otimes g^* i_* \cO_{K_n} )\\
&&\hspace{1cm}\trm{since }\pi_{13}\circ r\simeq \pi_{21}\circ(\id_A\times q\times f)\circ(q_{23}\times\pr_{12})\trm{ and }g\circ \pr_{23}\simeq\pi_2\circ r\\
&\simeq& \pi_{21*}(\id_A\times q\times f)_*(\id_A\times g)^*(q\times f)_*(\cO(E) \otimes g^* i_* \cO_{K_n} )\\
&&\hspace{1cm}\trm{by base change \cite[Appendix A]{addington2011new}}\\
&\simeq& \pi_{21*}(\id_A\times q\times f)_*(\id_A\times g)^*(q\times f)_*(\cO(E) \otimes (q\times f)^* \cO_{J} )\\
&&\hspace{1cm}\trm{since }g^* i_* \cO_{K_n}\simeq (q\times f)^*\cO_J\\
&\simeq& \pi_{21*}(\id_A\times q\times f)_*(\id_A\times g)^*((q\times f)_*\cO(E)\otimes \cO_{J})\quad\trm{by projection formula}\\
&\simeq& \pi_{21*}(\id_A\times q\times f)_*(\id_A\times g)^*\cO_{J}\quad\trm{since }(q\times f)_*\cO(E)\simeq\cO_{A\times A^{[n]}}\\
&\simeq& \pi_{21*}(\id_A\times q\times f)_*(\id_A\times g)^*h^*\cO_e\quad\trm{since }J:=h^{-1}(e)\\
&\simeq& \pi_{21*}(\id_A\times q\times f)_*(\id_A\times q\times f)^*(\id_A\times h)^*\Sigma^*\cO_e\\
&&\hspace{1cm}\trm{since }h\circ(\id_A\times g)\simeq(\Sigma\circ(\id_A\times h))\circ(\id_A\times q\times f)\\
&\simeq& \pi_{21*}((\id_A\times h)^*\Sigma^*\cO_e\otimes(\id\times q\times f)_*\cO_{A\times A^{[n-1,n]}})\quad\trm{by projection formula}\\
&\simeq& \pi_{21*}(\id_A\times h)^*\Sigma^*\cO_e\quad\trm{since }(\id\times q\times f)_*\cO_{A\times A^{[n-1,n]}}\simeq\cO_{A\times A\times A^{[n-1]}}\\
&\simeq& \pi_{21*}\cO_{J'}
\end{eqnarray*}
where \[J':=\Supp((\id_A\times h)^*\Sigma^*\cO_e)=\{(x,-x-m(\eta),\eta)\}\simeq A\times A^{[n-1]}.\]
As before, we can identify $\pi_{21}|_{J'}:J'\to A\times A$ with $-h\times\id_A:A^{[n-1]}\times A\to A\times A$ and the identity $\pi_{21*}\cO_{J'}\simeq\cO_{A\times A} \otimes H^*(K_{n-2},\cO_{K_{n-2}})$ follows from Lemma \ref{splitting} below.


Thus we have an exact triangle
$$ \cO_\Delta \otimes H^*(K_{n-1},\cO_{K_{n-1}}) \ra R_K''F_K'' \ra \cO_{A \times A} \otimes H^*(K_{n-2},\cO_{K_{n-2}}), $$
which must split because 
it is natural and hence (its defining morphism is) invariant under automorphisms of $A$; use the same argument as in \S\ref{HilbA}. 
 In other words, the kernel of this composition is given by \[R_K''F_K''=\left\{\begin{array}{c}\cO_\Delta\oplus\cO_\Delta[-2]\oplus\cdots\oplus\cO_\Delta[2-2n] \\\\ \oplus\;\cO_{A\times A}\oplus\cO_{A\times A}[-2]\oplus\cdots\oplus\cO_{A\times A}[4-2n]\end{array}\right.\]


 
 
(\textbf{7}) For the computation of $R_KF_K$, we exploit the following diagram of exact triangles
$$ \xymatrix{
R_K''F_K \ar[r] \ar[d] & R_K'F_K \ar[r] \ar[d] & R_KF_K \ar[d] \\
R_K''F_K' \ar[r] \ar[d] & R_K'F_K' \ar[r] \ar[d] & R_KF_K' \ar[d] \\
R_K''F_K'' \ar[r] & R_K'F_K'' \ar[r] & R_KF_K''
} $$
Taking cohomology of the left column and bottom row yields $$\cH^i(R_K''F_K)=\left\{\begin{array}{cc}\cO_\Delta & \trm{if }i=1,3,\ldots,2n-3\\ \cI_\Delta & \trm{if }i=2n-2\\ 0 & \trm{o/w}\end{array}\right.$$ and $$\cH^i(R_KF_K'')=\left\{\begin{array}{cc}\cO_{A\times A} & \trm{if }i=-2\\ \cO_\Delta & \trm{if }i=-1,1,\ldots,2n-3\\ 0 & \trm{o/w}\end{array}\right.$$ Then we can take cohomology of either the top row or right column to get $$\cH^i(R_KF_K)=\left\{\begin{array}{cc}\cO_\Delta & \trm{if }i=0,2,\ldots,2n-2\\ 0 & \trm{o/w}\end{array}\right.$$ Thus, $R_KF_K$ has a filtration whose associated graded object is $\bigoplus_{i=0}^{n-1}\cO_\Delta[-2i]$, but since classes in $\Ext^1_{A \times A}(\cO_\Delta, \cO_\Delta) \simeq HH^1(A)\simeq H^1(A,\cO_A)\oplus H^0(A,\cT_A)$ are invariant under automorphisms of $A\times A$, the filtration splits; use the same argument as in \S\ref{HilbA} with $\iota$ replaced by $\iota\times\iota$. (One also needs to check that $\iota\times\iota$ acts on $\Ext^3_{A\times A}(\cO_\Delta,\cO_\Delta)$ as $-1$ but this follows from Serre duality 
together with the fact that a volume form on $A\times A$ is invariant under $(\iota\times\iota)^*$ because we have even dimension.) Therefore, the kernel is $$R_KF_K=\cO_\Delta\oplus\cO_\Delta[-2]\oplus\cdots\oplus\cO_\Delta[2-2n].$$


 The monad structure $R_KF_KR_KF_K\srel{\mu}{\ra} R_KF_K$ is the hard part of the proof but again, we can appeal to \cite[Section 2.5]{addington2011new}. Indeed, the adjoint pairs $g_*\dashv g^!$ and $i^*\dashv i_*$ provide us with a map of monads $q_*q^*\to q_*g^!i_*i^*g_*q^*=R_K''F_K''$. Combining this with our observation that $q_*g^*i_*\cO_{K_n}\simeq\pi_{1*}\cO_J$, we can identify the monad structure of $R_K''F_K''$ with the multiplication in $H^*(K_{n-1},\cO_{K_{n-1}})$. Producing a map $q_*q^*\to R_KF_K$ (where we implicitly understand that $q:=q|_{(q\times f)^*J}$) goes through in exactly the same way as \cite[p.26-31]{addington2011new}. Thus, we can conclude with the following

\begin{thm}\label{kummer}
 Let $K_n$ be the generalised Kummer variety associated to an abelian surface $A$ and consider the natural functor $F_K:\cD(A)\ra\cD(K_n)$ induced by the universal ideal sheaf $\cI_{\bar{\cZ}_{n+1}}$ on $A\times K_n$. Then $F_K$ is a $\bbP^{n-1}$-functor for all $n>1$. In particular, we have a non-trivial derived autoequivalence \[P_{F_K}:=\cone(\cone(F_KR_K[-2]\to F_KR_K)\to\id_{K_n})\in\Aut(\cD(K_n))\]
\end{thm}

\section{The Albanese Map}\label{Albmap}
 Recall that the Albanese map $m:A^{[n]}\ra A$ is isotrivial. That is, we have the following cartesian diagram:
 $$\xymatrix{A\times K_{n-1} \ar[r]^-{\nu} \ar[d]_-{\pi_1} & A^{[n]} \ar[d]^-{m} & \left(x,\zeta\right) \ar@{|->}[r]^-{\nu} \ar@{|->}[d]_-{\pi_1} & t_x\zeta \ar@{|->}[d]^-{m} \\
 A \ar[r]_-{\varphi} & A & x \ar@{|->}[r]_-{\varphi} & nx }$$
 Explicitly, we have $A\times_A A^{[n]}=\left\{(x,\zeta)\in A\times A^{[n]}\;|\;m\zeta=nx\right\}$ and an isomorphism $A\times_A A^{[n]} \srel{\sim}{\lra} A\times K_{n-1}\;;\;(x,\zeta) \mapsto (x,t_{-x}\zeta)$. In other words, the morphism $\nu$ is just the restriction of the translation map on $A^{[n]}$ to $K_{n-1}$.

\begin{lem}\label{splitting}
 $m_*\cO_{A^{[n]}}\simeq\bigoplus_{i=0}^{n-1}\cO_A[-2i]$ is a formal object in $\cD(A)$.
\end{lem}

\begin{proof}
 Semicontinuity implies that $\cH^{2i}(m_*\cO_{A^{[n]}})=:\cL_i$ is locally free of rank one and $\cH^{2i+1}(m_*\cO_{A^{[n]}})=0$ for all $0\leq i\leq n-1$. Beauville \cite[Proposition 7]{beauville1983varietes} shows that the natural symplectic structure on $A^{[n]}$ induces a symplectic structure on the fibres $K_{n-1}$ of $m$. In particular, powers of the symplectic form provide nowhere vanishing sections of the $\cL_i$. Therefore, $\cL_i\simeq\cO_A$ for all $0\leq i\leq n-1$ and $m_*\cO_{A^{[n]}}$ has a filtration in cohomology sheaves which must split because $\Ext_A^{2k+1}(\cO_A,\cO_A)\simeq H^{2k+1}(A,\cO_A)=0$ for all $k\geq1$.
%
%
\end{proof}

\begin{thm}\label{Albanese}
 $m^*:\cD(A)\ra\cD(A^{[n]})$ is a $\bbP^{n-1}$-functor with $H=[-2]$.
\end{thm}

\begin{proof}
 First, we note that
 \begin{eqnarray*}
 m_*m^*\cE &\simeq& \cE\otimes m_*\cO_{A^{[n]}} \quad\trm{by projection formula}\\
 & \simeq & \bigoplus_{i=0}^{n-1}\cE[-2i]\quad\trm{by Lemma \ref{splitting}}.
 \end{eqnarray*}
Next, we need to check that $m_*m^*[-2]\hra m_*m^*m_*m^*\xra{m_*\epsilon m^*}m_*m^*$ induces an isomorphism on $\cH^i$ for $2 \le i \le 2n-2$ but this can be pointwise, in which case it follows immediately from Lemma \ref{splitting}. Finally, the condition $S_{A^{[n]}}m^*[2-2n]\simeq m^*S_A$ is satisfied because $S_{A^{[n]}}\simeq [2n]$ and $S_A\simeq[2]$.
\end{proof}

\section{Equivariant Approach}\label{equiv}
In this section we provide an alternative proof of \cite[Theorem 2]{addington2011new} and our Theorem \ref{kummer} above which makes use of the BKR-Haiman equivalence together with results of Scala \cite{scala2009cohomology} and Krug \cite{krug2014extension}. It should be noted that this proof of \cite[Theorem 2]{addington2011new} is very similar to an unpublished proof by Markman and Mehrotra \cite{markman2011integral}; 
the main difference is that they compute $RF$ directly, by calculating equivariant $\Tor$'s and using a standard vanishing Lemma in derived categories, whereas Krug's formulae \cite[Theorem 3.17]{krug2014extension} (which work for \emph{any} smooth projective surface) together with the methods therein, allow us to deduce $RF$ from all the other compositions $R'F'$, $R''F'$, $R'F''$ and $R''F''$. Also, the case of the generalised Kummer variety follows more easily from our approach.

Haiman \cite[Theorem 5.1]{haiman2001hilbert} proved that the Bridgeland-King-Reid theorem \cite[Theorem 1.1]{bridgeland2001mckay} applied to the Hilbert scheme $X^{[n+1]}$ of points on a smooth projective surface $X$ yields a derived equivalence $\Psi:\cD(X^{[n+1]})\xra{\sim}\cD^{\frS_{n+1}}(X^{n+1})$ between the bounded derived category of coherent sheaves on the Hilbert scheme and the bounded derived category of $\frS_{n+1}$-equivariant sheaves on the product.

This equivalence has been used extensively in the study of tautological objects on Hilbert schemes, that is, objects which lie in the image of our Fourier-Mukai functor $F'':\cD(X)\ra\cD(X^{[n+1]})$. In particular, we have the following 
\begin{thm}\label{extformula}
 Let $X$ be a smooth projective surface and consider the tautological objects $F''(\cE), F''(\cF)\in\cD(X^{[n+1]})$ associated to $\cE,\cF\in\cD(X)$. Then we have the following natural isomorphisms of graded vector spaces
 \begin{eqnarray*}
 \Ext^*(F''(\cE),F''(\cF)) &\simeq& \Ext^*(\cE,\cF)\otimes S^nH^*(X,\cO_X)\\ && \oplus\; H^*(X,\cE^\vee)\otimes H^*(X,\cF) \otimes S^{n-1}H^*(X,\cO_X)\\
 \Ext^*(F''(\cE),F'(\cF)) &\simeq& H^*(X,\cE^\vee)\otimes H^*(X,\cF) \otimes S^nH^*(X,\cO_X)\\
  \Ext^*(F'(\cE),F''(\cF)) &\simeq& H^*(X,\cE^\vee\otimes\omega_X)[2]\otimes H^*(X,\cF) \otimes S^nH^*(X,\cO_X)\\
   \Ext^*(F'(\cE),F'(\cF)) &\simeq& H^*(X,\cE^\vee\otimes\omega_X)[2]\otimes H^*(X,\cF) \otimes S^{n+1}H^*(X,\cO_X).
 \end{eqnarray*}
\end{thm}

\begin{proof}
 See \cite[Theorem 3.17 and Remark 3.19]{krug2014extension} for the first two. 
 The third expression follows from \cite[Corollary 35]{scala2009some} and \cite[Remark 3.20]{krug2014extension}.
\end{proof}

 This result allows us to completely circumvent the main calculation in \cite[Sections 2.2 \& 2.3]{addington2011new}. More precisely, using adjunctions and the Yoneda lemma, we can determine the kernels of the four crucial compositions $R'F',R''F',R'F'',R''F''$ almost instantly. Indeed, if $\Phi_\cP$ denotes the Fourier-Mukai transform with kernel $\cP$ then $\Phi_{\cO_X\boxtimes\cO_X}(\cF)\simeq H^*(\cF)\otimes\cO_X$, $\Phi_{\cO_X\boxtimes\omega_X}(\cF)\simeq H^*(\cF)\otimes\omega_X$ and $\Phi_{\cO_\Delta}(\cF)\simeq\cF$. Thus, we see that $\Ext^*(\cE,\Phi_{\cO_X\boxtimes\cO_X}(\cF))\simeq H^*(\cE^\vee)\otimes H^*(\cF)$, $\Ext^*(\cE,\Phi_{\cO_X\boxtimes\omega_X}(\cF))\simeq H^*(\cE^\vee\otimes\omega_X)\otimes H^*(\cF)$ and $\Ext^*(\cE,\Phi_{\cO_\Delta}(\cF))\simeq \Ext^*(\cE,\cF)$. Therefore, by Theorem \ref{extformula}, we have the following identifications:
 \begin{table}[htbp!]
 \centering
 \resizebox{\columnwidth}{!}{
  \begin{tabular}{| c | c |}
    \hline
     & \\
    \multirow{2}{*}{$R''F' \simeq \cO_{X\times X}\otimes H^*(X^{[n]},\cO_{X^{[n]}})$} & \multirow{2}{*}{$R'F' \simeq (\cO_X\boxtimes \omega_X)\otimes H^*(X^{[n+1]},\cO_{X^{[n+1]}})[2]$} \\
     & \\
     & \\
    \hline
    \multirow{2}{*}{$R''F'' \simeq \cO_{\Delta}\otimes H^*(X^{[n]},\cO_{X^{[n]}})$} & \\
     & \multirow{2}{*}{$R'F'' \simeq (\cO_X\boxtimes \omega_X)\otimes H^*(X^{[n]},\cO_{X^{[n]}})[2]$} \\
    \hspace{2.1cm}\multirow{2}{*}{$\oplus\; \cO_{X\times X} \otimes H^*(X^{[n-1]},\cO_{X^{[n-1]}})$} & \\
     & \\
    \hline
  \end{tabular}}
 \end{table}

 As discussed in \S\ref{HilbA}, the natural map $F'\ra F''$ of functors is induced by the restriction map $\cO_{X\times X^{[n+1]}}\ra \cO_{\cZ_{n+1}}$ of kernels. Under the BKR isomorphism, this is equivalent to the restriction map $\cO_{X\times X^{n+1}}\ra \cO_{D}$ where $D:=\bigcup_{i=1}^{n+1}D_i$ is the union of all pairwise diagonals $D_i:=\Delta_{i,n+2}\subset X^{n+1}\times X$. In the notation introduced above, this is just saying that $\Psi F'\simeq\Phi_{\cO_{X\times X^{n+1}}}$ and $\Psi F''\simeq\Phi_{\cO_{D}}$; see \cite[Proposition 1.3.3 \& section 2.1]{scala2009cohomology} 
 for more details. By \cite[Remark 2.2.1]{scala2009cohomology}, we have a \v{C}ech-type $\frS_{n+1}$-equivariant 
 resolution $\cK^\bullet$ of $\cO_{D}$ given by 
 \begin{equation}\label{res}0\ra\cO_{D}\ra \bigoplus_{i=1}^{n+1}\cO_{D_i}\ra\cdots\ra\bigoplus_{|I|=k}\cO_{D_I}\ra\cdots\ra \cO_{D_{\{1,\ldots,n+1\}}}\ra0\end{equation}
where $D_I:=\bigcap_{i\in I}D_i$ is the partial diagonal corresponding to $\emptyset\neq I\subseteq\{1,\ldots,n+1\}$. Now, the main point of Theorem \ref{extformula} and its proof \cite[Proposition 3.12]{krug2014extension} is that only the zeroth term of this complex contributes to $\cH om(\Psi(F'(\cF)),\Psi(F''(\cF)))^{\frS_{n+1}}$. More precisely, \cite[Theorem 2.4.5]{scala2009cohomology} and \cite[Proposition 3.12]{krug2014extension} provide the following identifications 
\begin{eqnarray*}
\cH om(\cO_{X\times X^{n+1}},\Psi(F''(\cF)))^{\frS_{n+1}} &\simeq& \cH om(\cO_{X\times X^{n+1}},C(\cF))^{\frS_{n+1}}\\
\cH om(\Psi(F''(\cE)),\cO_{X\times X^{n+1}})^{\frS_{n+1}} &\simeq& \cH om(C(\cE),\cO_{X\times X^{n+1}})^{\frS_{n+1}}\\
\cH om(\Psi(F''(\cE)),\Psi(F''(\cF)))^{\frS_{n+1}} &\simeq& \cH om(C(\cE),C(\cF))^{\frS_{n+1}}
\end{eqnarray*}
where 
\[C:=\Phi_{\bigoplus_{i=1}^{n+1}\cO_{D_i}}\simeq \bigoplus_{i=1}^{n+1}\pi_i^*:\cD(X)\to\cD^{\frS_{n+1}}(X^{n+1}).\]
Replacing $\cO_{D}$ by $\cK^\bullet$, we see that our restriction map must factor through $\cO_{X\times X^{n+1}}\\\to\bigoplus_{i=1}^{n+1}\cO_{D_i}$ which is nothing but restriction componentwise. Therefore, the induced map $\Psi(F'(\cF))\ra\Psi(F''(\cF))$ is realised as the sum of $n+1$ evaluation maps 
 \begin{equation*}H^*(\cF)\otimes\cO_X^{\boxtimes(n+1)}\to \bigoplus_i\pi_i^*\cF\;;\;f\otimes(s_1\boxtimes\cdots\boxtimes s_{n+1})\mapsto \bigoplus_i s_1\boxtimes\cdots \boxtimes f(s_i)\boxtimes\cdots \boxtimes s_{n+1}\end{equation*} 
%
 Thus, the induced map $\Ext^*(F'(\cE),F'(\cF))\ra\Ext^*(F'(\cE),F''(\cF))$, under the isomorphism of Theorem \ref{extformula}, is given by 
 \begin{align*}
 H^*(\cE)^\vee\otimes H^*(\cF)\otimes S^{n+1}H^*(\cO_X) &\ra H^*(\cE)^\vee\otimes H^*(\cF)\otimes S^{n}H^*(\cO_X)\\ e\otimes f\otimes(s_1\otimes\cdots \otimes s_{n+1}) &\mapsto \sum_i e\otimes f(s_i)\otimes s_1\otimes\cdots \otimes \what{s_i}\otimes\cdots \otimes s_{n+1}
 \end{align*}
 and similarly for the map $\Ext^*(F''(\cE),F'(\cF))\ra\Ext^*(F''(\cE),F''(\cF))$. By duality, we get similar descriptions for the maps $\Ext^*(F''(\cE),F'(\cF))\ra\Ext^*(F'(\cE),F'(\cF))$ and $\Ext^*(F''(\cE),F''(\cF))\ra\Ext^*(F'(\cE),F''(\cF))$ respectively; so it is enough to understand the first one. 
 
 In particular, if $X$ is a K3 surface then each $2k^\trm{th}$ graded piece of $S^{n}H^*(X,\cO_X)$ 
 is one-dimensional for $0\leq k\leq n$ and the map described above must be an isomorphism. Indeed, if $0\neq u\in H^2(X,\cO_X)$ then a basis of $S^nH^*(X,\cO_X)$ is given by elements of the form $u^k\id^{n-k}$ for $k=0,\ldots,n$ and the components 
 \[H^*(\cE)^\vee\otimes H^*(\cF)\otimes u^k\id^{n+1-k} \to H^*(\cE)^\vee\otimes H^*(\cF)\otimes u^k\id^{n-k}\] of $\Ext^*(F'(\cE),F'(\cF))\ra\Ext^*(F'(\cE),F''(\cF))$ described above are isomorphisms. 
By duality, the components 
of $\Ext^*(F''(\cE),F'(\cF))\ra\Ext^*(F'(\cE),F'(\cF))$ must also be isomorphisms and similarly for the other maps. This means we can cancel these terms from the cones and rewrite the following diagram of triangles $$\xymatrix@=0.3em{\Ext^*(F''(\cE),F(\cF)) \ar[rr] \ar[dd] && \Ext^*(F'(\cE),F(\cF)) \ar[rr] \ar[dd] && \Ext^*(F(\cE),F(\cF)) \ar[dd]\\ \\ \Ext^*(F''(\cE),F'(\cF)) \ar[rr] \ar[dd] && \Ext^*(F'(\cE),F'(\cF)) \ar[rr] \ar[dd] && \Ext^*(F(\cE),F'(\cF)) \ar[dd]\\ \\ \Ext^*(F''(\cE),F''(\cF)) \ar[rr] && \Ext^*(F'(\cE),F''(\cF)) \ar[rr] && \Ext^*(F(\cE),F''(\cF))}$$
 as
 $$\resizebox{\columnwidth}{!}{\xymatrix@=0.3em{{\begin{array}{c}\bigoplus_{k=0}^n \Ext^*(\cE,\cF)[-2k-1]\\ \oplus H^*(\cE^\vee)\otimes H^*(\cF)[-2n]\end{array}} \ar[rr] \ar[dd] && H^*(\cE^\vee)\otimes H^*(\cF)[-2n] \ar[rr] \ar[dd] && \bigoplus_{k=0}^n\Ext^*(\cE,\cF)[-2k] \ar[dd]\\ \\ \bigoplus_{k=0}^n H^*(\cE^\vee)\otimes H^*(\cF)[-2k] \ar[rr] \ar[dd] && \bigoplus_{k=-1}^n H^*(\cE^\vee)\otimes H^*(\cF)[-2k] \ar[rr] \ar[dd] && H^*(\cE^\vee)\otimes H^*(\cF)[2] \ar[dd]\\ \\ {\begin{array}{c}\bigoplus_{k=0}^n\Ext^*(\cE,\cF)[-2k]\\ \bigoplus_{k=0}^{n-1} H^*(\cE^\vee)\otimes H^*(\cF)[-2k]\end{array}} \ar[rr] && \bigoplus_{k=-1}^{n-1} H^*(\cE^\vee)\otimes H^*(\cF)[-2k] \ar[rr] && {\begin{array}{c}\bigoplus_{k=0}^n\Ext^*(\cE,\cF)[-2k+1]\\ \oplus H^*(\cE^\vee)\otimes H^*(\cF)[2]\end{array}}}}$$
 That is, $\Ext^*(\cE,RF(\cF))\simeq\Ext^*(F(\cE),F(\cF))\simeq\bigoplus_{k=0}^n\Ext^*(\cE,\cF)[-2k]$ and so by Yoneda, we have $$RF\simeq\bigoplus_{k=0}^n\cO_\Delta[-2k]$$ 
 \bigskip
 
 Similar arguments when $X$ is an abelian surface recover the expressions of \S\ref{HilbA}. For the case of the generalised Kummer variety, we play the same game with a `BKR-type' isomorphism $\Psi_K:\cD(K_n)\xra{\sim}\cD^{\frS_{n+1}}(N)$ where $j:N\hra A^{n+1}$ is the locus of points which sum to zero. Consider the following diagram:
 \[\xymatrix{
 & I^{n+1}A \ar[rr]^(0.3)p\ar@{->}'[d]^-q[dd] & & A^{n+1} \ar[dd]^(0.3)\pi\\
 K_n\times N \ar@{^{(}->}[ur]^-k\ar[rr]^(0.75){\bar{p}}\ar[dd]^(0.3){\bar{q}} & & N \ar@{^{(}->}[ur]^-j\ar@{->}[dd]^(0.3){\bar{\pi}}\\
 & A^{[n+1]} \ar@{->}'[r]^(0.5)\mu[rr] & & S^{n+1}A\\
 K_n \ar@{^{(}->}[ur]^-i\ar[rr]^(0.7){\bar{\mu}} & & N/\frS_{n+1} \ar@{^{(}->}[ur]_-\ell
 }\]
 
 \begin{lem}\label{BKRtype}
There is a `BKR-type' equivalence $\Psi_K:\cD(K_n)\xra\sim\cD^{\frS_{n+1}}(N)$ which naturally intertwines with $\Psi:\cD(A^{[n+1]})\xra\sim\cD^{\frS_{n+1}}(A^{n+1})$. That is, we have 
\[\Psi_K\circ i^*\simeq j^*\circ\Psi.\]
In particular, this provides us with the identities $\Psi_KF_K'\simeq\Phi_{\cO_{A\times N}}$, $\Psi_KF_K''\simeq\Phi_{\bar{D}}$ as well as an $\frS_{n+1}$-equivariant resolution $\bar{\cK}^\bullet$ of $\bar{D}:=(\id_A\times j)^{-1}D$ given by \[0\to\cO_{\bar{D}}\to\bigoplus_{i=1}^{n+1}\cO_{\bar{D}_i}\to\cdots\ra\bigoplus_{|I|=k}\cO_{\bar{D}_I}\to\cdots\to\cO_{\bar{D}_{\{1,\ldots,n+1\}}}\to0\] where $\bar{D}_I:=(\id_A\times j)^{-1}(D_I)$ are the restricted diagonals.
 \end{lem}
 \begin{proof}
Let $\Psi_K:\cD(K_n)\to\cD^{\frS_{n+1}}(N)$ be the Fourier-Mukai functor induced by $k^*\cO_{I^{n+1}A}$ where $I^{n+1}A:=(A^{[n+1]}\times_{S^nA} A^{n+1})_\mrm{red}$ 
is Haiman's isospectral Hilbert scheme; the BKR isomorphism $\Psi$ is given by $\Phi_{\cO_{I^{n+1}A}}:\cD(A^{[n+1]})\xra\sim\cD^{\frS_{n+1}}(A^{n+1})$. Applying \cite[Lemma 6.1]{chen2002flops} we see that \[\Psi\circ i_*\simeq j_*\circ \Psi_K.\] Furthermore, by \cite[Proposition 6.2]{chen2002flops}, we know that $\Psi_K$ must be an equivalence and so taking left adjoints proves the first claim. For the second identity, recall from \S\ref{genkum} that $F_K''\simeq i^*F''$ and so $\Psi_K F_K''\simeq \Psi_K i^*F''\simeq j^*\Psi F''\simeq j^*\Phi_{\cO_{D}}$. Now, since $\codim_{D}\bar{D}=2=\codim_{A\times A^{n+1}}A\times N$, we can use base change \cite[Corollary 2.27]{kuznetsov2006hyperplane} to conclude that $j^*\Phi_{\cO_{D}}\simeq \Phi_{\cO_{\bar{D}}}$; the first identity follows from similar arguments. Next, observe that $\codim_{D_I}\bar{D}_I$ and so by \cite[Lemma 2]{scala2009some} (or again by \cite[Corollary 2.27]{kuznetsov2006hyperplane}) we see that all the objects of \eqref{res} are $(\id_A\times j)^*$-acyclic and so the sequence remains exact after pulling back.
 \end{proof}
 \begin{rmk}
 We expect that the $\Psi_K$ defined above agrees with the BKR isomorphism induced by the universal family of $\frS_{n+1}$-clusters on $N$, or equivalently, the restricted isospectral Hilbert scheme $\overline{I^{n+1}A}:=q^{-1}(K_n)$ where $q:I^{n+1}A\to A^{[n+1]}$, but we are unable to give a formal proof of this `fact'. 
 \end{rmk}
 
 Since $\cH om$ is compatible with pullback \cite[Compatibilities (vi) p.85]{huybrechts2006fourier}, it follows immediately from \cite[Theorem 2.4.5]{scala2009cohomology} and \cite[Proposition 3.12]{krug2014extension} that we have natural isomorphisms
\begin{eqnarray*}
\cH om(\cO_{A\times N},\Psi_K(F_K''(\cF)))^{\frS_{n+1}} &\simeq& \cH om(\cO_{A\times N},\bar{C}(\cF))^{\frS_{n+1}}\\
\cH om(\Psi_K(F_K''(\cE)),\cO_{A\times N})^{\frS_{n+1}} &\simeq& \cH om(\bar{C}(\cE),\cO_{A\times N})^{\frS_{n+1}}\\
\cH om(\Psi_K(F_K''(\cE)),\Psi_K(F_K''(\cF)))^{\frS_{n+1}} &\simeq& \cH om(\bar{C}(\cE),\bar{C}(\cF))^{\frS_{n+1}}
\end{eqnarray*}
where 
\[\bar{C}:=\Phi_{\bigoplus_{i=1}^{n+1}\cO_{\bar{D}_i}}\simeq \bigoplus_{i=1}^{n+1}\bar{\pi}_i^*:\cD(A)\to\cD^{\frS_{n+1}}(N)\] 
and $\bar{\pi}_i:=\pi_i\circ j:N\subset A^{n+1}\to A$ are the natural projections restricted to $N$. This, in turn, gives rise to natural isomorphisms 
\begin{eqnarray*}
\Ext^*(\Psi_K(F_K'(\cE)),\Psi_K(F_K''(\cF)))^{\frS_{n+1}} &\simeq& \Ext^*(\Psi_K(F_K'(\cE)),\bar{C}(\cF))^{\frS_{n+1}}\\
\Ext^*(\Psi_K(F_K''(\cE)),\Psi_K(F_K'(\cF)))^{\frS_{n+1}} &\simeq& \Ext^*(\bar{C}(\cE),\Psi_K(F_K'(\cF)))^{\frS_{n+1}}\\
\Ext^*(\Psi_K(F_K''(\cE)),\Psi_K(F_K''(\cF)))^{\frS_{n+1}} &\simeq& \Ext^*(\bar{C}(\cE),\bar{C}(\cF))^{\frS_{n+1}}.
\end{eqnarray*}
In particular, it is only the zeroth term of the $\frS_{n+1}$-equivariant resolution $\bar{\cK}^\bullet$ which contributes to $\cH om(\Psi_K(F_K''(\cE)),\Psi_K(F_K''(\cF)))^{\frS_{n+1}}$ and the other compositions. Applying the arguments of \cite[Theorem 3.17]{krug2014extension} provides the following identities
\begin{eqnarray}
\Ext^*(\Psi_K(F_K'(\cE)),\Psi_K(F_K''(\cF))) &\simeq& H^*(A,\cE^\vee)\otimes \Ext^*(\cO_N,\bar{\pi}_{n+1}^*(\cF))^{S_n}\label{F'F''}\\
\Ext^*(\Psi_K(F_K''(\cE)),\Psi_K(F_K'(\cF))) &\simeq& H^*(A,\cF)\label{F''F'}\otimes \Ext^*(\bar{\pi}_{n+1}^*(\cE),\cO_N)^{S_n}\\
\Ext^*(\Psi_K(F_K''(\cE)),\Psi_K(F_K''(\cF))) &\simeq& \Ext^*(\bar{\pi}_{n+1}^*(\cE),\bar{\pi}_{n+1}^*(\cF))^{S_n}\label{F''F''} \\&& \oplus\; \Ext^*(\bar{\pi}_{n+1}^*(\cE),\bar{\pi}_{n}^*(\cF))^{S_{n-1}}.\nonumber
 \end{eqnarray}
 
Now, at this stage of the proof in the Hilbert scheme case, Krug uses the natural $S_n$-equivariant isomorphism $A^{n+1}\simeq A^n\times A$ together with the K\"unneth formula to obtain
$$\resizebox{\columnwidth}{!}{ $\Ext^*(\pi_{n+1}(\cE),\pi_{n+1}(\cF))^{S_n}\simeq H^*(A^n,\cO_{A^n})^{S_n}\otimes \Ext^*(\cE,\cF)\simeq H^*(A^{[n]},\cO_{A^{[n]}})\otimes \Ext^*(\cE,\cF).$}$$
However, if $N_n\subset A^{n+1}$ denotes the locus of points which sum to zero then the lack of an $S_n$-equivariant isomorphism $N_n\simeq N_{n-1}\times A$ prevents us from doing the same. To remedy this, we observe that pullback along the summation map $\Sigma_n:A^{n}\to A$ coincides with pullback along the Albanese map $m:A^{[n]}\to A$ under the BKR equivalence $\Psi:=\Phi_{\cO_{I^{n}A}}$; this allows us to induce Lemma \ref{splitting} and finish the proof.
 
 Consider the following diagram 
 \[\xymatrix{ && I^{n}A \ar[drr]^-p \ar[dll]_-q &&\\
  A^{[n]} \ar[drr]^-\mu\ar[dddrr]_-m \ar@{}[ddrr]|(0.75)\circlearrowleft &&\circlearrowleft&& A^{n}\ar[dll]_-\pi\ar[dddll]^-{\Sigma_n} \ar@{}[ddll]|(0.75)\circlearrowleft \\ 
 && S^{n}A \ar[dd]^-\tau &&\\
  && &&\\
  && A && }\]
 where $\tau:S^{n}A\to A$ denotes the summation map so as not to conflict with $\Sigma_n$.

\begin{lem}\label{summation}
$\Psi\circ m^*\simeq\Sigma_n^*:\cD(A)\to\cD^{\frS_{n}}(A^{n})$.
\end{lem}

\begin{proof}
By \cite[Proposition 1.3.3]{scala2009cohomology}, we have $p_*\cO_{I^{n}A}\simeq\cO_{A^{n}}$ and so $p_*p^*\simeq\id_{A^n}$. The claim now follows from the commutativity of the previous diagram:
\[\Psi\circ m^*\simeq p_*q^*\mu^*\tau^*\simeq p_*p^*\pi^*\tau^*\simeq \pi^*\tau^*\simeq\Sigma_n^*.\qedhere\]
\end{proof}

\begin{cor}\label{firstsummand}
$\Sigma_{n*}\Sigma_n^*\simeq \id_A\otimes H^*(K_{n-1},\cO_{K_{n-1}})$.
\end{cor}

\begin{proof}
By Lemma \ref{splitting}, we have $m_*m^*\simeq\id_A\otimes H^*(K_{n-1},\cO_{K_{n-1}})$. Combining this with Lemma \ref{summation}, we get
\[\Sigma_{n*}\Sigma_n^* \simeq m_*\Psi^{-1}\Psi m^* \simeq m_*m^* \simeq \id_A\otimes H^*(K_{n-1},\cO_{K_{n-1}}).\qedhere\]
\end{proof}

%

\begin{rmk}
To be formally correct, $\Psi$ should be replaced with $\Phi_{\cO_{I^{n}A}}\circ \mrm{triv}$ where $\mrm{triv}:\cD(A)\to\cD^{\frS_{n}}(A)$ equips every object with the trivial $\frS_{n}$-linearisation. Then $\Sigma_n^*$ becomes $\Sigma_n^*\circ\mrm{triv}$ and $\Sigma_{n*}$ becomes $(\_)^{\frS_{n}}\circ\Sigma_{n*}$. Recall that taking invariants $(\_)^{\frS_{n}}$ is right adjoint to $\mrm{triv}$ and $(\_)^{\frS_{n}}\circ\mrm{triv}\simeq\id_A$.
\end{rmk}

\begin{rmk}
Notice that Lemma \ref{summation} also shows that $\Sigma_n^*:\cD(A)\to\cD^{S_n}(A^n)$ is a $\bbP^{n-1}$-functor with twist $P_{\Sigma_n^*}\simeq \Psi\circ P_{m^*}\circ \Psi^{-1}$. Indeed, this follows immediately from Theorem \ref{Albanese} and the definitions in \S\ref{defs}; see \cite[Lemma 2.3]{krug2013new}.
\end{rmk}

Now, if we use the $S_n$-equivariant isomorphism \[A^n\simeq N\;;\;(x_1,\ldots,x_n)\mapsto (x_1,\ldots,x_n,-\sum_{i=1}^nx_i)\] then we see immediately that $\bar{\pi}_{n+1}$ can be identified (up to sign) with $\Sigma_n$. In particular, Corollary \ref{firstsummand} says that
\[(\bar{\pi}_{n+1*}\cO_N)^{S_n}\simeq(\Sigma_{n*}\cO_{A^n})^{S_n}\simeq \cO_A\otimes H^*(K_{n-1},\cO_{K_{n-1}})\simeq \bigoplus_{i=0}^{n-1}\cO_A[-2i]\] 
which allows us to simplify equations \eqref{F'F''}, \eqref{F''F'} and the first summand of \eqref{F''F''}. For the second summand of \eqref{F''F''}, we have

\begin{cor}\label{secondsummand}
$(\bar{\pi}_{n+1*}\bar{\pi}_{n}^*\cF)^{S_{n-1}}\simeq \cO_A\otimes H^*(A,\cF)\otimes H^*(K_{n-2},\cO_{K_{n-2}})$.
\end{cor}

\begin{proof}
Observe that $\bar{\pi}_n$ coincides with $\pi_n$ under the isomorphism $A^n\simeq N$ described above and $\Sigma_n\simeq \Sigma_2\circ(\Sigma_{n-1}\times\id_A)$. Then 
\begin{align*}
(\bar{\pi}_{n+1*}\bar{\pi}_{n}^*\cF)^{S_{n-1}} &\simeq (\Sigma_{n*}\pi_n^*\cF)^{S_{n-1}}\\
& \simeq (\Sigma_{2*}(\Sigma_{n-1}\times\id_A)_*(\cO_{A^{n-1}}\boxtimes \cF))^{S_{n-1}} \\
& \simeq \Sigma_{2*}((\Sigma_{n-1*}\cO_{A^{n-1}})^{S_{n-1}} \boxtimes \cF) \\
& \simeq \Sigma_{2*}(\cO_A\boxtimes \cF) \otimes H^*(K_{n-2},\cO_{K_{n-2}})\quad\trm{by Corollary \ref{firstsummand}}\\
& \simeq \Sigma_{2*}\pi_2^* \cF \otimes H^*(K_{n-2},\cO_{K_{n-2}})\\
&\simeq \cO_A\otimes H^*(A,\cF)\otimes H^*(K_{n-2},\cO_{K_{n-2}})
\end{align*}
where the last line uses the fact that $\Sigma_{2*}\pi_2^*\simeq\Phi_{\cO_{A\times A}}$. Indeed, $\Gamma_{\pi_2}\times A$ and $A\times \Gamma_{\Sigma_{2}}$ intersect transversally in the subvariety $\{(b,a,b,a+b)\}\simeq A\times A$ and so the claim follows from \cite[Proposition 5.10]{huybrechts2006fourier}. 
\end{proof}

Thus, we have proved the following

 \begin{thm}\label{extformulakummer}
 Let $A$ be an abelian surface and consider the tautological objects $F_K''(\cE), F_K''(\cF)\in\cD(K_n)$ associated to $\cE,\cF\in\cD(A)$. Then we have the following natural isomorphisms of graded vector spaces
 \begin{eqnarray*}\Ext^*(F_K''(\cE),F_K''(\cF)) &\simeq& \Ext^*(\cE,\cF)\otimes H^*(K_{n-1},\cO_{K_{n-1}})\\ && \oplus\; H^*(A,\cE^\vee)\otimes H^*(A,\cF) \otimes H^*(K_{n-2},\cO_{K_{n-2}})\\
\Ext^*(F_K''(\cE),F_K'(\cF)) &\simeq& H^*(A,\cE^\vee)\otimes H^*(A,\cF) \otimes H^*(K_{n-1},\cO_{K_{n-1}})\\
\Ext^*(F_K'(\cE),F_K''(\cF)) &\simeq& H^*(A,\cE^\vee)[2]\otimes H^*(A,\cF) \otimes H^*(K_{n-1},\cO_{K_{n-1}})\\
\Ext^*(F_K'(\cE),F_K'(\cF)) &\simeq& H^*(A,\cE^\vee)[2]\otimes H^*(A,\cF) \otimes H^*(K_{n},\cO_{K_{n}}).
 \end{eqnarray*}
\end{thm}



As before, we can use adjunctions and the Yoneda lemma to determine the kernels of the four compositions:
\begin{table}[htbp!]
 \centering
 \resizebox{\columnwidth}{!}{
  \begin{tabular}{| c | c |}
    \hline
     & \\
    \multirow{2}{*}{$R_K''F_K' \simeq \cO_{A\times A}\otimes H^*(K_{n-1},\cO_{K_{n-1}})$} & \multirow{2}{*}{$R_K'F_K' \simeq \cO_{A\times A}\otimes H^*(K_n,\cO_{K_n})[2]$} \\
     & \\
     & \\
    \hline
    \multirow{2}{*}{$R_K''F_K'' \simeq \cO_{\Delta}\otimes H^*(K_{n-1},\cO_{K_{n-1}})$} & \\
     & \multirow{2}{*}{$R_K'F_K'' \simeq \cO_{A\times A}\otimes H^*(K_{n-1},\cO_{K_{n-1}})[2]$} \\
    \hspace{2.1cm}\multirow{2}{*}{$\oplus\; \cO_{A\times A} \otimes H^*(K_{n-2},\cO_{K_{n-2}})$} & \\
     & \\
    \hline
  \end{tabular}}
 \end{table}

 

Similar arguments to those in the Hilbert scheme case show that the induced map $\Psi_K(F_K'(\cF))\ra\Psi_K(F_K''(\cF))$ is again realised as the sum of evaluation maps. Thus, 
the induced map $\Ext^*(F_K'(\cE),F_K'(\cF))\ra\Ext^*(F_K'(\cE),F_K''(\cF))$, under the isomorphism of Theorem \ref{extformulakummer}, is given by
 \begin{align*}
 H^*(\cE)^\vee\otimes H^*(\cF)\otimes H^*(\cO_N)^{\frS_{n+1}} &\ra H^*(\cE)^\vee\otimes H^*(\cF)\otimes H^*(\cO_{N_{n-1}})^{\frS_{n}}\\ e\otimes f\otimes(s_1\otimes\cdots \otimes s_{n+1}) &\mapsto \sum_i e\otimes f(s_i)\otimes s_1\otimes\cdots \otimes \what{s_i}\otimes\cdots \otimes s_{n+1}
 \end{align*}
 and similarly for the other maps. Now, if $0\neq v\in H^2(A,\cO_A)$ then a basis of the $\frS_{n+1}$-invariants of $H^*(N,\cO_N)\simeq H^*(A^n,\cO_{A^n})\simeq H^*(A,\cO_A)^{\otimes n}$ is given by elements of the form $v^k\id^{n-k}$ for $k=0,\ldots,n$ and we see that the components 
\[H^*(\cE)^\vee\otimes H^*(\cF)\otimes v^k\id^{n-k} \to H^*(\cE)^\vee\otimes H^*(\cF)\otimes v^k\id^{n-k-1}\] 
of $\Ext^*(F'(\cE),F'(\cF))\ra\Ext^*(F'(\cE),F''(\cF))$ are again isomorphisms. By duality, the components of the other maps are isomorphisms as well and so we can cancel the direct summands from the cones just as before and use Yoneda to get 
\[R_KF_K\simeq\bigoplus_{k=0}^{n-1}\cO_\Delta[-2k].\]
 
 The last piece of this technical jigsaw is the monad structure. Inspecting the proof of \cite[Theorem 3.17]{krug2014extension}, we can see that the first summand $\bigoplus_{k=0}^n\Ext^*(\cE,\cF)[-2k]$ of $\Ext^*(F''(\cE),F''(\cF))$ corresponds precisely to the summand $\Ext^*(\pi_{n+1}^*\cE,\pi_{n+1}^*\cF)^{\frS_{n}}$ of $\Ext^*(C(\cE),C(\cF))^{\frS_{n+1}}:=\Ext^*(\bigoplus_i\pi_i^*\cE, \bigoplus_i\pi_i^*\cF)^{\frS_{n+1}}$. 
 Thus, the adjoint pair $\pi_{n+1}^*\dashv\pi_{n+1*}$ allows us to identify $RF$ with $\pi_{n+1*}\pi_{n+1}^*\simeq
 \id_X\otimes H^*(X^{[n]},\cO_{X^{[n]}})$ where the monad structure is given by cup product. Similarly, if we replace $\pi_i$ by $\bar{\pi}_i:=\pi_i\circ j:N\to A$ where $N\simeq A^n$, then we can use $\bar{\pi}_{n+1}^*\dashv\bar{\pi}_{n+1*}$ to identify $R_KF_K$ with $\bar{\pi}_{n+1*}\bar{\pi}_{n+1}^*\simeq
 \id_A\otimes H^*(K_{n-1},\cO_{K_{n-1}})$ which follows from Corollary \ref{firstsummand}.
\medskip

 In summary, we have demonstrated independent proofs of \cite[Theorem 2]{addington2011new} and Theorem \ref{kummer} above on the equivariant sides of the BKR equivalences.
 
\section{Final Remarks}
By \cite[Theorem 0.2]{yoshioka2001moduli}, we know that the moduli space $\cK:=\cK_H(v)$ of $H$-stable sheaves on an abelian surface with primitive Mukai vector $v$, trivial determinant and trivial determinant of the Fourier-Mukai transform (with respect to the Poincar\'e bundle) is deformation equivalent to $K_{v^2/2-1}$ and so we expect Theorem \ref{kummer} to hold in much more generality. More precisely, if $F:\cD(A)\ra\cD(\cK)$ is the Fourier-Mukai functor induced by the universal sheaf $\cU$ on $A\times\cK_H(v)$ and $R$ is its right adjoint then we expect the kernel of $RF$ to be given by $\bigoplus_{i=0}^{\dim\cK-1}\cO_\Delta[-2i]$ and the monad structure $RFRF\srel{\mu}{\ra} RF$ to be like multiplication in the graded ring $H^*(\bbP^{\dim\cK-1},\bbC)$; cf. \cite[Conjecture on p.2]{addington2011new}. This would provide new autoequivalences of $\cD(\cK)$.
 
A first step towards establishing this conjecture can be taken by using an idea of Markman \& Mehrotra \cite{markman2011integral}. In particular, their arguments allow us to show that Mukai's homomorphism $\theta_v:v^\perp\xra\sim H^2(\cK_H(v),\bbZ)$ \cite[Theorem 0.2]{yoshioka2001moduli} is equivalent to C\u ald\u araru's $F^\dagger:HH_2(\cK)\xra\sim HH_2(A)$ \cite[Definition 5.2]{caldararu2003mukai} which factors through $\Hom(RF,\id_A)\xra{\circ\eta}HH_2(A)$ and hence splits the unit map $\eta:\id_A\to RF$. That is, $\cO_\Delta$ is a direct summand of the kernel of $RF$ which implies $F$ is faithful. We expect it should be possible to `bootstrap' this argument and determine a complete description of $RF$ in line with the suggestion above. 
 

Apart from Hilbert schemes of points $S^{[n]}$ on K3 surfaces $S$ and generalised Kummer varieties $K_n$ associated to abelian surfaces $A$ (together with deformations thereof), the only other compact hyperk\"ahler varieties known to us are O'Grady's sporadic examples of dimension ten and six; see \cite{OGrady1999desingularized} and \cite{OGrady2003new}. Lehn \& Sorger \cite{lehn2006la} show that if $v$ is a primitive Mukai vector with $v^2=2$ then the moduli space $\cM:=\cM_H(2v)$ (resp. $\cK:=\cK_H(2v)$) of $H$-stable sheaves on a K3 surface $S$ (resp. abelian surface $A$) admits a symplectic resolution $\pi:\widetilde{\cM}\to\cM$ (resp. $\pi:\widetilde{\cK}\to\cK$) which is obtained by blowing up the (reduced) singular locus $\Sym^2\cM_H(v)$ (resp. $\Sym^2\cK_H(v)$). In this case, Perego \& Rapagnetta \cite{perego2013deformation} show that $\widetilde\cM$ and $\widetilde\cK$ are deformation equivalent to O'Grady's ten and six dimensional example respectively. It is natural at this point to ask if we can construct $\bbP^n$-functors for the O'Grady spaces $\widetilde\cM$ and $\widetilde\cK$. 


The hyperk\"ahler Strominger-Yau-Zaslow conjecture \cite[Conjecture 1.2]{verbitsky2010hyperkahler} states that every hyperk\"ahler manifold can be deformed into a hyperk\"ahler manifold which admits a lagrangian fibration. Therefore, one could also investigate whether there is a natural $\bbP^n$-functor associated to a lagrangian fibration $\pi:X\to\bbP^n$.
 
\bibliographystyle{alpha}
\bibliography{ref}
\end{document}